\documentclass[a4paper,nosmartrunhead]{svjour3}       
\smartqed  

\usepackage{etex}
\usepackage[utf8x]{inputenc}
\usepackage{pgf,tikz}
\usetikzlibrary{arrows}

\usepackage{comment}

\newtheorem{assumption}{Assumption}

\usepackage[english]{babel}

\usepackage[boxruled,linesnumbered]{algorithm2e}

\newcommand{\assign}{\leftarrow}
\newcommand{\ARCTAN}{\mathop{\mathrm{atan2}}}

\usepackage{graphics,graphicx}
\usepackage{bm,bbm}
\usepackage{multicol}

\usepackage[math]{easyeqn}
\usepackage{easyvector}
\eqrowsep{3pt plus 2pt minus 2pt}

\newcommand{\dt}{\,\mathrm{d}t}
\newcommand{\ds}{\,\mathrm{d}s}

\newcommand{\pp}{\bm{p}}
\newcommand{\qq}{\bm{q}}

\renewcommand{\aa}{\bm{a}}

\newcommand{\bb}{\bm{b}}
\newcommand{\cc}{\bm{c}}

\newcommand{\FLOAT}[1]{\mathrm{fl}\left(#1\right)}
\newcommand{\ANGLE}[1]{\mathcal{A}\left(#1\right)}
\newcommand{\umach}{{\color{red}u}}

\newcommand{\DOT}{\!\cdot\!}
\newcommand{\CROSS}{\times}

\newcommand{\tgamma}{\widetilde{\bm\gamma}}
\newcommand{\talpha}{\widetilde{\bm\alpha}}
\newcommand{\tbeta}{\widetilde{\bm\beta}}
\newcommand{\tomega}{\widetilde{\bm\omega}}

\newcommand{\tL}{\widetilde{L}}

\newcommand{\ttheta}{\widetilde{\theta}}
\newcommand{\tsigma}{\widetilde{\sigma}\,}
\newcommand{\tx}{\widetilde{x}\,}
\newcommand{\ty}{\widetilde{y}\,}
\newcommand{\hx}{\widehat{x}\,}
\newcommand{\hy}{\widehat{y}\,}

\newcommand{\doi}[1]{\ensuremath{#1}}

\begin{document}

\title{Efficient computation of Linking number with certification}

\author{Enrico Bertolazzi \and Riccardo Ghiloni \and Ruben Specogna}

\institute{%
Enrico Bertolazzi \at
Department of Industrial Engineering -- University of Trento, Italy \\
\email{enrico.bertolazzi@unitn.it}
\and
Riccardo Ghiloni \at
Department of Mathematics -- University of Trento, Italy \\
\email{ghiloni@science.unitn.it}
\and
Ruben Specogna \at
Dipartimento Politecnico di Ingegneria ed Architettura -- University of Udine, Italy \\
\email{Ruben.Specogna@uniud.it}
}
\date{}

\maketitle

\begin{abstract}
  An efficient numerical algorithm for the computation of linking number 
  is presented. The algorithm keep tracks or rounding error so that 
  it can ensure the correctness of the results.
\end{abstract}

\keywords{Linking Number \and Computational Topology}



\section{Linking and writhe number as angles summation}
The Linking number of two closed parametric continuous curves
\begin{EQ}
  \left\{
  \begin{array}{l}
    \pp:[0,N]\to\mathbbm{R}^3\\
    \qq:[0,M]\to\mathbbm{R}^3
  \end{array}
  \right.
  \qquad
  \left\{
  \begin{array}{l}
    \pp(0)=\pp(N)\\
    \qq(0)=\qq(M)
  \end{array}
  \right.
\end{EQ}
well separated, i.e.,
\begin{EQ}
  \delta = \min_{s\in[0,N], t\in[0,M]} \norm{\pp(s)-\qq(t)} > 0
\end{EQ}
is given by the double integral \cite{Ricca:2011,Oosterom:1983}:
\begin{EQ}\label{eq:L}
   L(\pp,\qq)= \dfrac{1}{4\pi}
   \int_{0}^N
   \int_{0}^M
   \dfrac{(\qq(s)-\pp(t))\DOT(\qq'(s)\CROSS\pp'(t))}
         {\norm{\qq(s)-\pp(t)}^3}\ds\dt\,.
\end{EQ}
Writhe number has a similar definition~\cite{Berger:2009,Agarwal:2004,Fuller:1971}
and is connected with linking number as $W(\pp)=L(\pp,\pp)$,
where the integral~\eqref{eq:L} for $W(\pp)$ becomes
singular and in this case the principal value must be considered.

In practical numerical computation we restrict curves 
to polygonal closed curves defined as:
\begin{EQ}[rcllcllcl]\label{eq:pq}
  \pp(t) &=& \pp_{i} + (t-i)(\pp_{i+1}-\pp_i) \quad &
  \textrm{for $t$} &\in&[i,i+1], \quad & 
  i&=&1,2\ldots,N \\
  \qq(s) &=& \qq_{j} + (s-j)(\qq_{j+1}-\qq_j) \quad &
  \textrm{for $s$}&\in&[j,j+1], \quad &
  j&=&1,2,\ldots,M \\
\end{EQ}
where $N$ and $M$ are the number of segments of the 
first and second curves, respectively.
Curves $\pp(t)$ and $\qq(t)$ are closed so that  
points $\pp_k$ and $\qq_k$ must satisfy $\pp_{N+1}=\pp_1$ and
$\qq_{M+1}=\qq_1$.
If the approximation of continuous curve with polygons
does not add additional crossing the Linking number does not change
while the Writhe number does not change too much if the
approximating curves are close to the original one,
see~\cite{Cantarella:2005}.

Integral~\eqref{eq:L} with polygonal curves~\eqref{eq:pq} can be broken 
as the sum of $NM$ integrals:
\begin{EQ}[rcl]\label{eq:Lij:sum}
  L(\pp,\qq) &=& \dfrac{1}{2\pi} \sum_{i=1}^N\sum_{j=1}^M \Delta\Theta_{ij},\quad
  \Delta\Theta_{ij} = T(\pp_i,\pp_{i+1},\qq_j,\qq_{j+1}) \\
  T(\pp,\pp',\qq,\qq') &=& \dfrac{1}{2}
  \int_{0}^1
  \int_{0}^1
  \dfrac{(\qq-\pp + s\Delta\qq-t\Delta\pp)
         \DOT(\Delta\pp\CROSS\Delta\qq)}
         {\norm{\qq-\pp + s\Delta\qq-t\Delta\pp}^3}\ds\dt\, \\
   \Delta\pp&=&\pp'-\pp \qquad
   \Delta\qq=\qq'-\qq
\end{EQ}
The summation of all the $\Delta\Theta_{ij}$ is an integer multiple of $2\pi$, this number is the Linking Number.
The angle $T(\pp,\pp',\qq,\qq')$ has the properties
\begin{enumerate}
  \item $T(\pp,\pp',\qq,\qq')=T(\qq,\qq',\pp,\pp')$;
  \item $T(\pp,\pp',\qq,\qq')=-T(\pp',\pp,\qq,\qq')$;
  \item $T(\pp,\pp',\pp,\qq')=0$;
  \item $T(\pp,\pp,\qq,\qq')=0$;
\end{enumerate}
property $2$ means that if one curve is reversed the linking number
change the sign. While properties $3$ and $4$ means that if two point touch the
angle is $0$.
Properties $3$ and $4$ are important for Writhe number and means 
that this number is stable for small deformation of the curve.
Property $5$ is useful in numerical computation because the
value can be scaled be


\section{Angle triples}
In the next section will be showed that the angles $\Delta\Theta_{ij}$ can be computed as $NM$ differences of 
two evaluation of the four quadrant function $\ARCTAN(y,x)$.
\begin{definition}\label{atan2:def}
  The value of the function $\ARCTAN(y,x)$ defined for $(x,y)\in\mathbbm{R}^2\setminus\{(0,0)\}$ 
  is the angle $\theta$ obtained as the unique solution of the problem
  \begin{EQ}\label{eq:atandef}
    \theta=\ARCTAN(y,x), \quad\textrm{$\theta$ solution of:}\quad
    \cases{
      R\cos\theta = x, & \\
      R\sin\theta = y, &
    }\textrm{with}\;
    \cases{
       R>0, & \\
       \theta\in(-\pi,\pi]. &
    }
  \end{EQ}
\end{definition}
The computation of $\ARCTAN(\cdot,\cdot)$ can be completely removed
in the computation of Linking number while is reduced to only one computation
for the computation of Writhe number.

Consider the vector $(x,y)^T$ then $\ARCTAN(y,x)$
is the angle of this vector respect to the $x-$axis.
Thus, instead of the computation of $\ARCTAN(y,x)+\ARCTAN(y',x')$
we can consider the vector $(x'',y'')^T$ obtained by turning $(x',y')^T$
by the angle $\theta=\ARCTAN(y,x)$, i.e.
\begin{EQ}\label{eq:theta:to:mat}
  \pmatrix{ x''\\ y''} = 
  \pmatrix{ \cos\theta & -\sin\theta \\ \sin\theta & \cos\theta }
  \pmatrix{ x'\\ y'},\qquad
  \theta = \ARCTAN(y,x).
\end{EQ}
With this relation we have the identity $\ARCTAN(y,x)+\ARCTAN(y',x')=\ARCTAN(y'',x'')$
unless  $\ARCTAN(y,x)+\ARCTAN(y',x')$ is greater than $\pi$ or less than $-\pi$.
In this latter case $2\pi$ must be added or subtracted.
Notice that by set $R=\sqrt{x^2+y^2}$ equation~\eqref{eq:theta:to:mat} can be written as
\begin{EQ}\label{eq:rot}
  \pmatrix{ x''\\ y''} = 
  \dfrac{1}{R}
  \pmatrix{ x & -y \\ y & x }
  \pmatrix{ x'\\ y'}=
  \dfrac{1}{R}
  \pmatrix{ xx'-yy'\\ xy'+yx'},
\end{EQ}
and thus observing that $\ARCTAN(y,x)=\ARCTAN(\alpha y,\alpha x)$
for all $\alpha>0$ we have that any angle $\theta$ can be represented 
by a triple $[x,y,\sigma]$ where $(x,y)^T\in\mathbbm{R}^2\setminus \{(0,0)\}$ and $\sigma\in\mathbbm{Z}$ 
as $\theta = \ARCTAN(y,x) + 2\pi\sigma$ and summation can be 
represented as matrix vector multiplication \eqref{eq:rot}.
This suggest the following definition:
\begin{definition}\label{triple:def}
  The triple $[x,y,\sigma]\in \left(\mathbbm{R}^2\setminus \{(0,0)\}\right)\times\mathbbm{Z}$ is mapped to the angle $\theta$ 
  by the function $\mathcal{A}: \left(\mathbbm{R}^2\setminus \{(0,0)\}\right)\times\mathbbm{Z}\to\mathbbm{R}$
  defined as
  \begin{EQ}
     \theta = \ANGLE{[x,y,\sigma]} =\ARCTAN(y,x)+2\pi \sigma
  \end{EQ}
  two triple $[x,y,\sigma]$ and $[x',y',\sigma']$ are equivalent, 
  i.e $[x,y,\sigma]\equiv [x',y',\sigma']$,
  if $\ANGLE{[x,y,\sigma]}=\ANGLE{[x',y',\sigma']}$, i.e. when 
  the triples correspond to the same angle. Equivalently
  when $x'=\alpha x$, $y'=\alpha y$ and $\sigma'=\sigma$ for
  $\alpha>0$.
  
\end{definition}


\subsection{Angle triples summation and properties}
To compute efficiently angle summation in~\eqref{eq:Lij:sum}
an operative definition of triples summation is necessary.
For this definition the sign of a point in the plane
with the origin removed $\mathbbm{R}^2\setminus \{(0,0)\}$ will be used forward 
to detect crossing with the positive $x-$axes.
\begin{definition}[Point sign]\label{def:2}
  For a point $(x,y)$ in $\mathbbm{R}^2\setminus \{(0,0)\}$ 
  the sign $s(x,y)$ is the function 
  \begin{EQ}[rcl]\label{eq:sign}
    s(x,y) &=& \cases{
      +1 & if $(x,y)\in Q^+$,\\
      -1 & if $(x,y)\in Q^-$,
    } \\
    \textrm{where} &&
    \cases{
    Q^+ = \big\{ (x,y)\,|\, (y>0) \;\mathrm{or}\; ((y=0)\;\mathrm{and}\; (x<0)) \big\}, & \\
    Q^- = \big\{ (x,y)\,|\, (y<0) \;\mathrm{or}\; ((y=0)\;\mathrm{and}\; (x>0)) \big\}. &
    }
  \end{EQ}
  moreover the region $Q^+$ and $Q^-$ satisfy 
  $Q^+ \cup Q^-=\mathbbm{R}^2\setminus \{(0,0)\}$ and $Q^+ \cap Q^-=\emptyset$
  see Figure~\ref{fig:Qregion}.
\end{definition}
\begin{definition}[Cross detection]\label{def:3}
  Given the points $(x,y)$, $(x',y')$ and $(x'',y'')$
  in $\mathbbm{R}^2\setminus \{(0,0)\}$
  the cross detection function for the $x$-axes on the left half plane is defined as
  \begin{EQ}\label{eq:sigma}
    \mathcal{C}(x,y,x',y',x'',y'') = \cases{
      s(x,y) & if $s(x,y)s'>0$ and $s(x,y)s''< 0$, \\
      0 & otherwise. \\
    }
  \end{EQ}
  where $s'=s(x',y')$ and $s''=s(x'',y'')$.
  The values of the function are in the set 
  $\{0,1,-1\}$ where $0$ mean no cross $1$ crossing counter-clockwise 
  and $-1$ crossing clockwise.
  This function will be used in the next Lemma
  with sign function~\eqref{eq:sign} to detect 
  a turn around the origin of a 2D polygon.
\end{definition}
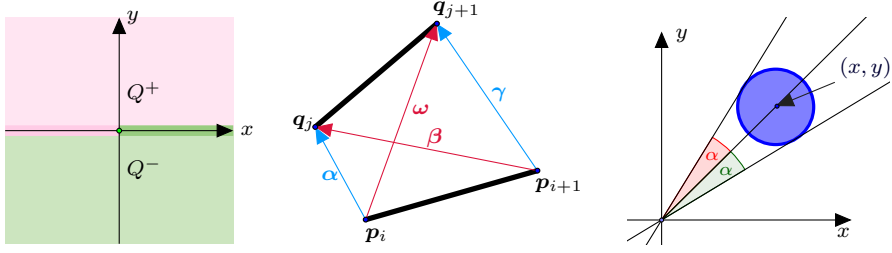
\begin{figure}[t]
    \begin{center}
    \definecolor{zzccff}{rgb}{0.6,0.8,0.5}
    \definecolor{ffcctt}{rgb}{1,0.8,0.9}
    \begin{tikzpicture}[>=triangle 45,x=1.0cm,y=1.0cm,scale=0.3]
      \fill[color=zzccff,fill=zzccff,fill opacity=0.5] (-5,-5) -- (5,-5) -- (5,0) -- (-5,0) -- cycle;
      \fill[color=ffcctt,fill=ffcctt,fill opacity=0.5] (-5,5) -- (-5,0) -- (5,0) -- (5,5) -- cycle;
      \draw [line width=4pt,color=ffcctt] (-5,0)-- (0,0);
      \draw [line width=4pt,color=zzccff] (0,0)-- (5,0);
      \draw [->] (-5,0) -- (5,0);
      \draw [->] (0,-5) -- (0,5);
      \draw (0,2.5)  node[anchor=north west ] {$Q^+$};
      \draw (0,-2.5) node[anchor=south west ] {$Q^-$};
      \draw (5,0)    node[anchor=west ] {$x$};
      \draw (0,5)    node[anchor=west ] {$y$};
      \draw [fill=green] (0,0) circle (3pt);
      \clip(-5,-5) rectangle (5,5);
    \end{tikzpicture}
    \definecolor{dcrutc}{rgb}{0.8627450980392157,0.0784313725490196,0.23529411764705882}
    \definecolor{qqzzff}{rgb}{0.0,0.6,1.0}
    \definecolor{qqqqff}{rgb}{0.0,0.0,1.0}
    \begin{tikzpicture}[line cap=round,line join=round,>=triangle 45,x=1.0cm,y=1.0cm,scale=0.6]
      \clip(-2.2,0.18) rectangle (4.4,5.75);
      \draw [line width=2.0pt] (-0.42519782090679226,0.7216273992108917)-- (3.339451241157656,1.8045491036354429);
      \draw [line width=2.0pt] (-1.530511337526607,2.7696331512405536)-- (1.1349683261595518,5.053717195937718);
      \draw (-0.6,0.7) node[anchor=north west] {$\pp_i$};
      \draw (3.131586984750401,1.8) node[anchor=north west] {$\pp_{i+1}$};
      \draw (-2.2,3.2) node[anchor=north west] {$\qq_j$};
      \draw (0.88,5.7) node[anchor=north west] {$\qq_{j+1}$};
      \draw [->,color=qqzzff] (-0.42519782090679226,0.7216273992108917) -- (-1.5305113375266068,2.7696331512405536);
      \draw [->,color=qqzzff] (3.339451241157656,1.8045491036354429) -- (1.1349683261595516,5.053717195937718);
      \draw [->,color=dcrutc] (3.339451241157656,1.8045491036354429) -- (-1.5305113375266073,2.7696331512405536);
      \draw [->,color=dcrutc] (-0.42519782090679226,0.7216273992108917) -- (1.134968326159552,5.053717195937718);
      \draw [fill=qqqqff] (-0.42519782090679226,0.7216273992108917) circle (1.5pt);
      \draw [fill=qqqqff] (3.339451241157656,1.8045491036354429) circle (1.5pt);
      \draw [fill=qqqqff] (-1.530511337526607,2.7696331512405536) circle (1.5pt);
      \draw [fill=qqqqff] (1.1349683261595518,5.053717195937718) circle (1.5pt);
      \draw[color=qqzzff] (-1.2,1.7006169754318154) node {$\bm\alpha$};
      \draw[color=qqzzff] (2.5,3.5) node {$\bm\gamma$};
      \draw[color=dcrutc] (1.0826393144503146,2.532074001060834) node {$\bm\beta$};
      \draw[color=dcrutc] (0.7856903767256644,3.155666770282598) node {$\bm\omega$};
    \end{tikzpicture}
    \definecolor{qqwuqq}{rgb}{0.0,0.39215686274509803,0.0}
    \definecolor{ffqqqq}{rgb}{1.0,0.0,0.0}
    \definecolor{sqsqsq}{rgb}{0.12549019607843137,0.12549019607843137,0.12549019607843137}
    \definecolor{qqqqtt}{rgb}{0.0,0.0,0.2}
    \definecolor{xdxdff}{rgb}{0.49019607843137253,0.49019607843137253,1.0}
    \definecolor{qqqqff}{rgb}{0.0,0.0,1.0}
    \begin{tikzpicture}[line cap=round,line join=round,>=triangle 45,x=1.0cm,y=1.0cm,scale=0.5]
    \clip(0.1,0.36) rectangle (7,6.37);
    \draw [line width=1.2pt,color=qqqqff,fill=qqqqff,fill opacity=0.5] (4.0,4.0) circle (1.0cm);
    \draw [shift={(1.0,1.0)},color=ffqqqq,fill=ffqqqq,fill opacity=0.15] (0,0) -- (45.0:2.5684082486587503) arc (45.0:58.63302222536642:2.5684082486587503) -- cycle;
    \draw [shift={(1.0,1.0)},color=qqwuqq,fill=qqwuqq,fill opacity=0.1] (0,0) -- (31.366977774633593:2.5684082486587503) arc (31.366977774633593:45.0:2.5684082486587503) -- cycle;
    \draw [->] (-3.0,1.0) -- (6.0,1.0);
    \draw [->] (1.0,-2.0) -- (1.0,6.0);
    \draw [domain=0.01164270296947321:7.365401057024001] plot(\x,{(--1.3743685418725535--2.1461490623970567*\x)/3.5205176042696102});
    \draw [domain=0.01164270296947321:7.365401057024001] plot(\x,{(-1.3743685418725535--3.5205176042696102*\x)/2.1461490623970567});
    \draw [domain=1.0:7.365401057024001] plot(\x,{(-0.0--3.0*\x)/3.0});
    \draw (5.42,1.0) node[anchor=north west] {$x$};
    \draw (1.17,6.28) node[anchor=north west] {$y$};
    \draw [color=qqqqtt](5.472889715907026,5.52569991597516) node[anchor=north west] {$(x,y)$};
    \draw [->,color=sqsqsq] (5.54802848403337,4.6487817766746655) -- (4.036477218076305,4.018968749192554);
    \begin{scriptsize}
    \draw [fill=xdxdff] (1.0,1.0) circle (1.5pt);
    \draw [fill=qqqqff] (4.036477218076305,4.018968749192554) circle (1.5pt);
    \draw[color=ffqqqq] (2.33,2.7) node {$\alpha$};
    \draw[color=qqwuqq] (2.7,2.3) node {$\alpha$};
    \end{scriptsize}
    \end{tikzpicture}
  \end{center}
  \caption{On the left regions $Q=\mathbbm{R}^2\setminus\{(0,0)\}$, $Q^+$ and $Q^-$.
  On the center the vectors $\bm\alpha$, $\bm\beta$, $\bm\gamma$ and $\bm\omega$ used in the 
  computation of $L_{ij}$ or $\theta_{ij}$.
  On the right the geometric problem for the estimation of angle error from
  the error on the triple $[x,y,\sigma]$.}
  \label{fig:Qregion}
\end{figure}

\begin{lemma}\label{lemma:2}
  Let be $(x,y)$ and $(x',y')$ two points in $\mathbbm{R}^2\setminus \{(0,0)\}$ and
  \begin{EQ}
    x'' = xx'-yy', \qquad 
    y'' = xy'+yx'
  \end{EQ}
  then the following identity is true
  \begin{EQ}\label{eq:addarctan}
     \ARCTAN(y,x)+\ARCTAN(y',x') = \ARCTAN(y'',x'')
     +2\pi\,\mathcal{C}(x,y,x',y',x'',y'')
  \end{EQ}
\end{lemma}
\begin{proof}
The classical identity~\cite{Abramowitz:1965} about inverse tangent
taking care on quadrant changes is:
\begin{EQ}
   \arctan u + \arctan v = \arctan \dfrac{u+v}{1-uv}+n\pi,\qquad
   n = \cases{
      0 & $uv \leq 1$ \\[-0.25em]
      1 & if $u>0$ \\[-0.25em]
      -1 & if $u<0$\\[-0.25em]
   }
\end{EQ}
using the definition of $\ARCTAN$ the equality simplify to
$\arctan u + \arctan v = \ARCTAN(u+v,1-uv)$ (see~\cite{Bradford:2002}) 
so that by set $u=y/x$ and $v=y'/x'$
\begin{EQ}[rcl]\label{eq:lem:1}
  \arctan\frac{y}{x}+\arctan\frac{y'}{x'}
  &=& \ARCTAN\left(\frac{xy'+x'y}{xx'},\frac{xx'-yy'}{xx'}\right) \\
  &=& \ARCTAN\left(xy'+x'y,xx'-yy'\right)+\sigma\pi,
\end{EQ}
with $\sigma\in\{0,+1,-1\}$, moreover
\begin{EQ}\label{eq:lem:2}
  \ARCTAN\left(y,x\right) = 
  \arctan\frac{y}{x}+\cases{
    0 & if $x\geq 0$ \\[-0.25em]
    -\pi & if $x<0$ and $y<0$ \\[-0.25em]
    +\pi & if $x<0$ and $y\geq 0$ \\[-0.25em]
  }
\end{EQ}
using~\eqref{eq:lem:1} and~\eqref{eq:lem:2}
\begin{EQ}
  \ARCTAN(y,x)+\ARCTAN(y',x')
  =\ARCTAN\left(xy'+x'y,xx'-yy'\right)+\sigma\pi,
\end{EQ}
and $\sigma\in\{-2,-1,0,1,2\}$. The change of values of $\sigma$ is done when $(x,y)$ and $(x',y')$ change
between $Q^+$ and $Q^-$ so that we have $4$ cases:
\begin{EQ}[llccc]
   \quad&\quad&\;\ARCTAN(y,x)\;&\;\ARCTAN(y',x')\;&\ARCTAN(y,x)+\ARCTAN(y',x')\\
   (A) & \cases{(x,y)\in Q^+ \\ (x',y')\in Q^+}
   & [0,\pi)  & [0,\pi)  & [0,2\pi) 
   \\
   (B) & \cases{(x,y)\in Q^+ \\ (x',y')\in Q^-} & [0,\pi)  & [-\pi,0) & [-\pi,\pi)\\
   (C) & \cases{(x,y)\in Q^- \\ (x',y')\in Q^+} & [-\pi,0) & [0,\pi)  & [-\pi,\pi)\\
   (D) & \cases{(x,y)\in Q^- \\ (x',y')\in Q^-} & [-\pi,0) & [-\pi,0) & [-2\pi,0)
\end{EQ}
and the resulting $\sigma$:
\begin{EQ}[ll]
  (A) \qquad  & \kappa=
   \cases{
    0 & angle in $[0,\pi)$ \\
    2 & angle in $[\pi,2\pi)$ \\   
   }\\
  (B) & \kappa=0\\
  (C) & \kappa=0\\
  (D) & \kappa=
   \cases{
    -2 & angle in $[-2\pi,-\pi)$ \\
    0 & angle in $[-\pi,0)$ \\   
   }
\end{EQ}
thus, equation~\eqref{eq:addarctan} is true with $\sigma\in\{0,-1,1\}$.
If the points $(x,y)\in Q^+$ and $(x',y')\in Q^-$ 
or $(x',y')\in Q^+$ and $(x,y)\in Q^-$ 
summation cant be larger than $\pi$ and thus $\sigma = 0$.
If both $(x,y)\in Q^+$ and $(x',y')\in Q^+$ there is a crossing 
if  $(x'',y'')\in Q^-$ and thus $\sigma=+1$.
If both $(x,y)\in Q^-$ and $(x',y')\in Q^-$ there is a crossing 
if  $(x'',y'')\in Q^+$ and thus $\sigma=-1$.
This changes are resumed in function~\eqref{eq:sigma}.
\qed

\end{proof}
Lemma~\ref{lemma:2} suggest the following definition for the addition
of two triple:
\begin{definition}\label{triple:plus:def}
  The sum $[x,y,\sigma]\oplus[x',y',\sigma']$ of two triple
  $[x,y,\sigma]$ and $[x',y',\sigma']$ is the triple $[x'',y'',\sigma'']$
  defined as:
  \begin{EQ}\label{triple:plus:eq}
     \cases{
       x''= \alpha(xx'-yy'), &\\
       y''=\alpha(xy'+yx'), &\\
     }
     \sigma''= \sigma+\sigma'+\mathcal{C}(s(x,y),s(x',y'),s(x'',y'')),
     \quad
  \end{EQ}
  with $\alpha$ any positive real number and $\mathcal{C}(s(x,y),s(x',y'),s(x'',y''))$ defined in 
  equation~\eqref{eq:sigma}.
\end{definition}
For the triple introduced in definition~\ref{triple:def} with 
addition of definition~\ref{triple:plus:def} is trivial to 
prove the identities
\begin{EQ}[rcl]
   \ANGLE{[x,y,\sigma]\oplus[x',y',\sigma']}&=&\ANGLE{[x,y,\sigma]}+\ANGLE{[x',y',\sigma']},\\
   \ANGLE{[x,y,\sigma]\oplus[x,-y,-\sigma]}&=&\ANGLE{[1,0,0]},
\end{EQ}
which suggest the definition of the subtraction
\begin{EQ}
   [x,y,\sigma]\ominus[x',y',\sigma']:=[x,y,\sigma]\oplus[x',-y',-\sigma']
\end{EQ}
and the following identity
\begin{EQ}
   \ANGLE{[x,y,\sigma]\ominus[x',y',\sigma']}=\ANGLE{[x,y,\sigma]}-\ANGLE{[x',y',\sigma']},
\end{EQ}
easily follow.
Finally multiplication by a scalar $w\in\mathbbm{Z}$ can be defines as
\begin{EQ}
   w\cdot [x,y,\sigma] = \cases{
   [1,0,0] & for $w=0$ \\
    \underbrace{[x,y,\sigma]\oplus [x,y,\sigma] \oplus \cdots \oplus [x,y,\sigma]}_{\textrm{$w$ times}} & for $w>0$ \\
    \underbrace{[x,-y,-\sigma]\oplus [x,-y,-\sigma] \oplus \cdots \oplus [x,-y,-\sigma]}_{\textrm{$|w|$ times}} & for $w<0$
   }
\end{EQ}

\section{Linking number as summation of triples}

With the following lemma the contribution $\Delta\Theta_{ij}=T(\pp_i,\pp_{i+1},\qq_j,\qq_{j+1})$
can be written in term of a single triple $[x_{ij},y_{ij},\sigma_{ij}]$ as  $\Delta\Theta_{ij}=\ANGLE{[x_{ij},y_{ij},\sigma_{ij}]}$.
\begin{lemma}\label{lem:3}
  Let $\pp$, $\pp'$, $\qq$ and $\qq'$ such that
  the space segment $[\pp,\pp']$ and  $[\qq,\qq']$ \textbf{do not intersect}, then defining
  \begin{EQ}[rclrclrclrcl]\label{eq:abgm:def}
    \bm\alpha &=& \qq-\pp,\quad&
    \bm\beta  &=& \qq-\pp',\quad&
    \bm\gamma &=& \qq'-\pp',\quad&
    \bm\omega &=& \qq'-\pp,\\
    \talpha &=& \bm\alpha/\norm{\bm\alpha},\quad&
    \tbeta  &=& \bm\beta/\norm{\bm\beta},\quad&
    \tgamma &=& \bm\gamma/\norm{\bm\gamma},\quad&
    \tomega &=& \bm\omega/\norm{\bm\omega},
  \end{EQ}
  and
  \begin{EQ}\label{eq:xy}
    \cases{
      x\,\, = 1+\talpha\DOT\tgamma+\talpha\DOT\tbeta+\tbeta\DOT\tgamma,\\
      y\,\, = \tbeta\DOT(\talpha\CROSS\tgamma),
    }
    \qquad
    \cases{
      x' = 1+\talpha\DOT\tgamma+\talpha\DOT\tomega+\tomega\DOT\tgamma, \\
      y' = \tomega\DOT(\talpha\CROSS\tgamma), 
    }
  \end{EQ}
  it follows that $T(\pp,\pp',\qq,\qq')$ take the form
  \begin{EQ}\label{eq:sigma2}
     T(\pp,\pp',\qq,\qq')=\ANGLE{[x'',y'',\sigma'']},\\[1em]
     \sigma''=\cases{
       -s(x'',y'') & if $s(x'',y'')y>0$ \\
       0 & otherwise
    }
    \qquad
    \cases{
      x'' = xx'+yy', \\
      y'' = xy'-yx',
    }
  \end{EQ}
  Finally the points $(x,y)$, $(x',y')$ and $(x'',y'')$ are all different
  from $(0,0)$.
\end{lemma}

\begin{proof}
  If $\pp$, $\pp'$, $\qq$ and $\qq'$ are pairwise distinct then 
  all the formulae are well defined and a simple rewrite of the formulae in 
  \cite{ZinArai:2013} and \cite{Oosterom:1983}
  permits to write
  \begin{EQ}
     T(\pp,\pp',\qq,\qq') = \theta^+ - \theta^-,
     \qquad
     \cases{
       \theta^+ = \Theta(\qq-\pp,\qq-\pp',\qq'-\pp') \\
       \theta^- = \Theta(\qq-\pp,\qq'-\pp,\qq'-\pp')
     }
  \end{EQ}
  where
  \begin{EQ}[rcl]
    \Theta(\aa,\bb,\cc)&=&
    \ARCTAN\big( y(\aa,\bb,\cc), x(\aa,\bb,\cc)\big),
    \\
    x(\aa,\bb,\cc)&=&\norm{\aa}\norm{\bb}\norm{\cc}+
    (\bb\DOT\cc)\norm{\aa}+
    (\cc\DOT\aa)\norm{\bb}+
    (\aa\DOT\bb)\norm{\cc},
    \\
    y(\aa,\bb,\cc)&=&\aa\DOT(\bb\CROSS\cc).
  \end{EQ}
  notice that from $\ARCTAN\big( y, x ) = \ARCTAN\big( y/t, x/t )$ for any $t>0$ 
  it follows
  \begin{EQ}\label{eq:scale}
    \Theta(\aa,\bb,\cc)=\Theta(\aa/\norm{\aa},\bb/\norm{\bb},\cc/\norm{\cc})
  \end{EQ}
  Using~\eqref{eq:abgm:def} and definition~\ref{triple:plus:def} with~\eqref{eq:scale}
  \begin{EQ}[rcl]
    T(\pp,\pp',\qq,\qq') &=&
    \Theta(\bm\alpha,\bm\beta,\bm\gamma)-
    \Theta(\bm\alpha,\bm\omega,\bm\gamma)
    \\
    &=&
    \Theta(\bm\talpha,\bm\tbeta,\bm\tgamma)-
    \Theta(\bm\talpha,\bm\tomega,\bm\tgamma)
    \\
    &=&
    \ARCTAN\big(
      \talpha\DOT(\tbeta\CROSS\tgamma),\;
              1+
              \tbeta\DOT\tgamma+
              \tgamma\DOT\talpha+
              \talpha\DOT\tbeta
    \big)\\
    &-&
    \ARCTAN\big(
      \talpha\DOT(\tomega\CROSS\tgamma),\;
              1+
              \tomega\DOT\tgamma+
              \tgamma\DOT\talpha+
              \talpha\DOT\tomega
    \big)
    \\
    &=&
    \ARCTAN\big(
      -\tbeta\DOT(\talpha\CROSS\tgamma),\;
              1+\tbeta\DOT(\talpha+\tgamma)+
              \tgamma\DOT\talpha
    \big)\\
    &+&
    \ARCTAN\big(
              \tomega\DOT(\talpha\CROSS\tgamma),\;
              1+
              \tomega\DOT(\talpha+\tgamma)+
              \tgamma\DOT\talpha
    \big)
    \\
    &=&
    \ARCTAN(-y,x) + \ARCTAN(y',x') \\
    &=&\ANGLE{[x,-y,0]}+\ANGLE{[x',y',0]} \\
    &=&\ANGLE{[x'',y'',\sigma'']},
  \end{EQ}
  where $\sigma''=\mathcal{C}(s(x,-y),s(x',y'),s(x'',y''))$.
  From the identity $\bm\omega=\bm\gamma-\bm\beta+\bm\alpha$ it follows
  $\bm\beta\DOT(\bm\alpha\CROSS\bm\gamma)=-\bm\omega\DOT(\bm\alpha\CROSS\bm\gamma)$ 
  and thus $\textrm{sign}(\bm\tbeta\DOT(\bm\talpha\CROSS\bm\tgamma))=-\textrm{sign}(\bm\tomega\DOT(\bm\talpha\CROSS\bm\tgamma))$.
  Hence, from~\eqref{eq:xy} it follows $yy'<0$ and $s(x,-y)=-s(x',y')$ so that $\sigma''$ is given by~\eqref{eq:sigma2}.
  The point $(x,y)$ must satify $(x,y)\neq (0,0)$, on the contraty let be $x=y=0$ then
  \begin{EQ}\label{eq:impossible}
    \cases{
      0 = 1+\tgamma\DOT\talpha+\tbeta\DOT\talpha+\tbeta\DOT\tgamma,\\
      0 = \tbeta\DOT(\talpha\CROSS\tgamma),
    }
  \end{EQ}
  from the second equation we have two cases:
  \begin{itemize}
    \item[1)] $\talpha\CROSS\tgamma=\bm{0}$ and thus $\tgamma=t\talpha$ for some $t$ and from the first 
          equation $1+t+(1+t)\tbeta\DOT\talpha=0$ it follows or $t=-1$ or $\tbeta\DOT\talpha=-1$.
          But  $\tbeta\DOT\talpha=-1$ imply that $\qq_j$ lies on segment $[\pp_i,\pp_j]$,
          thus must be $t=-1$,
          so that $\tgamma+\talpha=0$
          and from~\eqref{eq:abgm:def}
    \begin{EQ}
        \tgamma+\talpha=0, \qquad \Rightarrow\qquad
        \dfrac{\norm{\bm{\alpha}}\qq'+\norm{\bm{\gamma}}\qq}
              {\norm{\bm{\alpha}}+\norm{\bm{\gamma}}}
         =
        \dfrac{\norm{\bm{\alpha}}\pp'+\norm{\bm{\gamma}}\pp}
              {\norm{\bm{\alpha}}+\norm{\bm{\gamma}}},
    \end{EQ}
    i.e. the segments $[\pp,\pp']$ and  $[\qq,\qq']$ intersect. 
    \item[2)] $\talpha\CROSS\tgamma\neq \bm{0}$ and thus $\tbeta$ is a linear combination 
          of $\tgamma$ and $\talpha$, i.e. $\tbeta = t(s\tgamma+(1-s)\talpha)$ for some $t$ and $s$.
          From the first 
          equation $1+\tgamma\DOT\talpha+t(1+\talpha\DOT\tgamma)=0$.
          So that or $t=-1$ or $\talpha\DOT\tgamma=-1$. 

          But  $\talpha\DOT\tgamma=-1$ imply that $\qq_j=\qq_{j+1}$ or $\pp_i=\pp_{i+1}$ 
          qith the segments that intersect,
          thus must be $t=-1$ and  $\tbeta = s(\talpha-\tgamma)-\talpha$.
          But $\norm{\tbeta}=1$ and the norm $\norm{s(\talpha-\tgamma)-\talpha}$ is a quadratic function in $s$
          equal to $1$ for $s=0$ and $s=1$. It follows or $\tbeta=-\talpha$ or
          $\tbeta=-\tgamma$. From~\eqref{eq:abgm:def}
          it follows that in both cases the segments $[\pp,\pp']$ and  $[\qq,\qq']$ intersect. 
  \end{itemize}
  Thus if the segment $[\pp,\pp']$ and  $[\qq,\qq']$ do not intersect 
  \eqref{eq:impossible} i.e. $x=y=0$ is never satisfied.
  Similar arguments are used to exclude $x'=y'=0$.
  \qed
\end{proof}

Lemma~\ref{lem:3} is used in procedure \textsf{buildAngle} of Table~\ref{tab:link}
which compute the angle $\Delta\Theta_{ij}$ as a triple $[x,y,\sigma]$ 
using the points $\pp_i$, $\pp_{i+1}$, $\qq_j$ and $\qq_{j+1}$.
Lemma~\ref{lem:3} is also the core for an efficient algorithm
for the computation of $L$  resumed in the following theorem:
\begin{theorem}\label{teo:4}
Giving $\Delta\Theta_{ij} = \ANGLE{[x_{ij},y_{ij},\sigma_{ij}]}$ 
computed using Lemma~\ref{lem:3} the summation
$[X,Y,\sigma]=\-\bigoplus_{i=1}^N\-\bigoplus_{j=1}^M\- [x_{ij},y_{ij},\sigma_{ij}]$
in exact arithmetics gives $X>0$, $Y=0$ and $\sigma=L$, the linking number, i.e. 
\begin{EQ}
   [X,0,L]= \bigoplus_{i=1}^N\bigoplus_{j=1}^M [x_{ij},y_{ij},\sigma_{ij}]
\end{EQ}

\end{theorem}
\begin{proof}
It follows trivially from Lemma~\ref{lemma:2}:
\begin{EQ}
  2\pi L =
  \sum_{i=1}^N\sum_{j=1}^M \ANGLE{[x_{ij},y_{ij},\sigma_{ij}]}
  =\ANGLE{ \bigoplus_{i=1}^N\bigoplus_{j=1}^M [x_{ij},y_{ij},\sigma_{ij}] }
\end{EQ}
where the last summation is done using formula~\eqref{triple:plus:eq}.
At the end of computation
\begin{EQ}
  \bigoplus_{i=1}^N\bigoplus_{j=1}^M [x_{ij},y_{ij},\sigma_{ij}] = [X,Y,\sigma]
  \qquad\Rightarrow\\
  2\pi L = \ANGLE{[X,Y,\sigma]}
  =\ARCTAN(Y,X)+2\pi\sigma
\end{EQ}
$L$ must be an integer so that $\ARCTAN(Y,X)$ must be a multiple of $\pi$ and thus
must be $\ARCTAN(Y,X)=0$. This imply $Y=0$ and $\sigma=L$.\qed
\end{proof}

\begin{table}[!tb]
\caption{Computation of linking number in \textbf{exact arithmetics}.
         The values $x_{ij}$, $y_{ij}$, $\sigma_{ij}$ and $\theta_{ij}$ 
         are defined in Lemma~\ref{lem:3}}
\label{tab:LINK:1}
\setlength{\columnsep}{0pt}
\begin{procedure}[H]
  \caption{link($\pp,\qq$)}
  $[X,Y,L]\assign [1,0,0]$\;
  \For{$i=1,2,\ldots,N$ \textbf{and} $j=1,2,\ldots,M$}{
    build angle $[x_{ij},y_{ij},\sigma_{ij}];\quad$
    $[X,Y,L]\assign[X,Y,L]\oplus[x_{ij},y_{ij},\sigma_{ij}]$\;
  }
  \Return{$L$}\;
\end{procedure}

\caption{Detailed procedure for the computation of the linking number 
         in exact arithmetics. Procedure \textsf{buildAngle} 
         compute the angle between two segment
         as a triple $[x,y,\sigma]$ ad a sign of $(x,y)$.
         Procedure \textsf{link} compute linking number 
         by accumulating angles stored as a triple $[X,Y,\ell]$.}
\label{tab:link}
\begin{procedure}[H]
  \caption{buildAngle($\pp,\pp',\qq,\qq'$)}
  \SetKw{OR}{or}
  \SetKw{AND}{and}
  $\bm\alpha \assign \qq-\pp\;$;\quad
  $\bm\beta  \assign \qq-\pp'\;$;\quad
  $\bm\gamma \assign \qq'-\pp'\;$;\quad
  $\bm\omega \assign \qq'-\pp\;$\;
  $\ell_\alpha \assign \norm{\bm\alpha}\;$;\quad
  $\ell_\beta  \assign \norm{\bm\beta}\;$;\quad
  $\ell_\gamma \assign \norm{\bm\gamma}\;$;\quad
  $\ell_\omega \assign \norm{\bm\omega}$\;
  $\bm\alpha \assign \bm\alpha/\ell_\alpha\;$;\quad
  $\bm\beta  \assign \bm\beta/\ell_\beta\;$;\quad
  $\bm\gamma \assign \bm\gamma/\ell_\gamma\;$;\quad
  $\bm\omega \assign \bm\omega/\ell_\omega\;$\;
  $\bm{t}_1  \assign \bm\alpha+\bm\gamma\;$;\quad
  $\bm{t}_2  \assign \bm\alpha\CROSS\bm\gamma\;$;\quad
  $t_3       \assign 1 + \bm\alpha\DOT\bm\gamma\;$\;
  $x_1       \assign t_3+\bm\beta\DOT\bm{t}_1\;$;\quad
  $y_1       \assign \bm\beta\DOT\bm{t}_2\;$;\quad
  $x_2       \assign t_3+\bm\omega\DOT\bm{t}_1\;$;\quad
  $y_2       \assign \bm\omega\DOT\bm{t}_2$\;
  $x         \assign x_1x_2+y_1y_2\;$;\quad
  $y         \assign x_1y_2-y_1x_2\;$\;
  \leIf{$y>0$ \OR ($y=0$ \AND $x<0$)}{$s\assign 1$}{$s\assign -1$}
  \leIf{$s\,y_1>0$}{$\sigma\assign-s$}{$\sigma\assign 0$}
  \Return{$x,y,\sigma,s$}
\end{procedure}
\vspace{-0.8em}
\begin{procedure}[H]
  \caption{link($\pp,\qq$)}
  \SetKwFunction{buildangle}{buildangle}
  \SetKwFunction{addangle}{addangle}
  \SetKwFunction{scaleangle}{scaleangle}
  \SetKw{OR}{or}
  \SetKw{AND}{and}
  $[X,Y,\ell]\assign [1,0,0];\quad S\assign-1$\;
  \For{$i=1,2,\ldots,N$}{
    \For{$j=1,2,\ldots,M$}{
      $x,y,\sigma,s\assign\buildAngle(\pp_i,\pp_{i+1},\qq_j,\qq_{j+1})$\;
      $x'\assign Xx-Yy\;$;\quad
      $y'\assign Xy+Yx\;$;\quad
      $\ell\assign\ell+\sigma\;$\;
      \leIf{$y'>0$ \OR ($y'=0$ \AND $x'<0$)}{$s'\assign 1$}{$s'\assign -1$}
      \lIf{$Ss>0$ \AND $Ss'<0$}{$\ell\assign\ell-s'$}      
      decompose $\cases{x'=m\cdot 2^e & \cr y'=n\cdot 2^f & }$ 
      with $\cases{\abs{m}\leq 1 & \cr \abs{n}\leq 1 & }$
      and set $h\assign\min\{e,f\}$\;
      $X\assign m\cdot 2^{e-h}\;$;\quad
      $Y\assign n\cdot 2^{f-h}\;$;\quad
      $S\assign s'\;$\;
    }
  }
  \Return{$\ell$}
\end{procedure}

\end{table}
Theorem~\ref{teo:4} suggest a simple algorithm for the computation of linking number which 
is sketched in Table~\ref{tab:LINK:1} while Table~\ref{tab:link}
contain the pseudocode of the complete algorithm.
In conclusion from Theorem~\ref{teo:4} in
the computation of linking number no $\ARCTAN(y,x)$ are necessary.
In practice, starting from $(X,Y)=(1,0)$
the application of formula~\eqref{triple:plus:eq} for each couple
of vector permits to obtain a final vector which angle is
the sum of all the angles.
There is the need to take track of rotation that cross 
negative $x$-axis, i.e. $y=0$ and $x<0$.
At the end of the computation using exact arithmetic
$(X,Y)=(1,0)$ and, thus, $\ARCTAN(0,1)=0$.
The linking number is given by counting (with sign) the number
of times the rotations cross the negative $x$-axis.

Using floating point arithmetics at the end of summation $(X,Y)=(d,\epsilon)$.
If accumulation error is limited $d>0$ and the ratio $\epsilon/d$ is expected small.
In any case there is no trivial way to guarantee that the accumulated error angle be less than $\pi/2$.
The important question when using floating point arithmetics 
is how many segments can be used for the computation of linking number
before accumulation error invalidate the result.

\section{Angle triples summation and subtraction with floating point noise}

Here computation of linking number in the presence of rounding error is considered.
For estimation the following model is assumed for floating point arithmetics.

\begin{assumption}\label{ass:1}

Floating-point numbers are assumed stored in the form $(sign)\cdot\mu\cdot2^\nu$, 
where $\nu$ is called the exponent, and $\mu$, with the value between 1 and 2, 
is called the mantissa.
For each input quantity $x$, $\FLOAT{x}$ denotes its floating-point approximation. 
In general $\FLOAT{\cdot}$ is the floating-point approximation for an operation. 
The approximations by input quantization and rounding and are modeled as the following~\cite{Higham:2002,Jeannerod:2013}:
\begin{EQ}\label{eq:fp}
  \FLOAT{x} = x(1+\beta),\quad
  \FLOAT{x\circ y} = (x\circ y)(1+\beta)= \dfrac{x\circ y}{1+\beta'},\\
  \FLOAT{\sqrt{x}} = \sqrt{x}\,(1+\beta)=\dfrac{\sqrt{x}}{1+\beta'},\quad
  \abs{\beta},\abs{\beta'} \leq \umach,
\end{EQ}
where $\circ= \{+, −, \cdot,/ \}$ and $\umach$ is the unit roundoff or machine epsilon.
Typically $\umach$ is less than $10^{-7}$ for single precision computation
and $10^{-15}$ for double precision computation. 
\textbf{In the following extimate $\umach \leq 10^{-7}$ is assumed.}
\end{assumption}

Due to rounding error the angles $\Delta\Theta_{ij} = \ANGLE{[x_{ij},y_{ij},\sigma_{ij}]}$
are approximated with $\widetilde{\Delta\Theta_{ij}} = \ANGLE{[\tx_{ij},\ty_{ij},\tsigma_{ij}]}$ which correspond to 
an approximated angle $\widetilde{\Delta\Theta_{ij}}=\Delta\Theta_{ij}+\delta\Theta_{ij}$.
Analogously, the summation of triples satisfy
\begin{EQ}
   \theta + \theta' = 
   \ANGLE{[x,y,\sigma]\oplus [x',y',\sigma']}
   =
   \ANGLE{[x'',y'',\sigma'']}=\theta'',
   \\{}
   [x'',y'',\sigma'']=[x,y,\sigma]\oplus [x',y',\sigma']
\end{EQ}
but in presence of rounding errors the computed triple $[x'',y'',\sigma'']$
is only approximated with $[\tx'',\ty'',\tsigma'']$ that satisfy
\begin{EQ}
   \theta'' + \delta\theta'' = 
   \ANGLE{[\tx'',\ty'',\tsigma'']},
\end{EQ}
where $\delta\theta''$ is the term due to floating point computation.
The floating point summation of the triples is thus equivalent 
to the following angles summation:
\begin{EQ}[rcl]
  \tL &=&
  \dfrac{1}{2\pi}
  \FLOAT{
  \sum_{i=1}^N\sum_{j=1}^M 
  \Delta\Theta_{ij}}
  =
  \dfrac{1}{2\pi}
  \sum_{i=1}^N\sum_{j=1}^M 
  \left(\Delta\Theta_{ij}+\delta\Theta_{ij}+\delta\theta_{ij}\right)\\
  &=& L + 
  \dfrac{1}{2\pi}\sum_{i=1}^N\sum_{j=1}^M \left(\delta\Theta_{ij}+\delta\theta_{ij}\right)
\end{EQ}
if the error in triple formation $\delta\Theta_{ij}$ can be bounded for
example $\abs{\delta\Theta_{ij}} \leq c_1$ and the errors in
triple summation $\delta\theta_{ij}$ can also be bounded for
example by $\abs{\delta\theta_{ij}} \leq c_2$ then
the difference of exact and approximated linking number satisfy:
\begin{EQ}\label{eq:c1:c2}
  \abs{L-\tL} \leq
  \dfrac{1}{2\pi}\sum_{i=1}^N\sum_{j=1}^M \left(\abs{\delta\theta_{ij}}+\abs{\delta\Theta_{ij}}\right)
  \leq \dfrac{NM}{2\pi}(c_1+c_2)
\end{EQ}

Thus, if $c_1$ and $c_2$ are known the condition $NM<\pi/(c_1+c_2)$
imply $\abs{L-\tL}<1/2$ and is a sufficient condition for exact computation 
of linking number with floating point arithmetics. 

An estimation of constants $c_1\leq 117.861 \umach$ is given by
Lemma~\ref{lem:14} when $\pp_i$, $\pp_{i+1}$, $\qq_j$ and $\qq_{j+1}$ satisfy
\begin{EQ}\label{eq:strict:cond}
  \norm{\pp_i-\pp_{i+1}}
  \leq C_{ij},
  \qquad
  \norm{\qq_j-\qq_{j+1}}
  \leq
  C_{ij},
  \\
  C_{ij} =
  \min\left\{
       \norm{\qq_j-\pp_i},
       \norm{\qq_j-\pp_{i+1}},
       \norm{\qq_{j+1}-\pp_i},
       \norm{\qq_{j+1}-\pp_{i+1}}
  \right\}
\end{EQ}
while constant $c_2 \leq 2.829 \umach$ by Lemma~\ref{lem:8}.
Condition~\eqref{eq:strict:cond} is not too restrictive because
can be obtained by splitting the segments. However is more practical
to use a posteriori error given by Lemma~\ref{lem:13}:
\begin{EQ}\label{eq:dynamic:cond}
     \abs{\delta\Theta_{ij}}\leq 
     \left( 2.829 + 
     \dfrac{57.516}{R}+
     \dfrac{57.516}{R'}
     \right) \umach,
     \\
     R = \FLOAT{\sqrt{x^2+y^2}},\quad
     R' = \FLOAT{\sqrt{(x')^2+(y')^2}},
\end{EQ}
where $(x,y)$ and $(x',y')$ are the intermediate values 
of formula~\eqref{eq:xy}.
Combining estimate~\eqref{eq:dynamic:cond} with estimate by Lemma~\ref{lem:8}
permits to compute an upper bound of the uncertainty of the linking number
after its numerical computation.

\begin{table}[!tb]
\caption{Detailed procedure for the computation of the linking number 
         in floating point arithmetics. Procedure \textsf{buildAngle} 
         compute the angle between two segment
         as a triple $[x,y,\sigma]$, a sign of $(x,y)$
         and an upper bound for the error as a multiple of machine epsilon.
         Procedure \textsf{link} compute linking number 
         by accumulating angles stored as a triple $[X,Y,\ell]$.
         It also accumulate error as  a multiple of machine epsilon $\umach$
         split in integer part $I$ and rational part $R$.
         When $I\geq I_{\max}$ accumulation of floating point errors
         may have greater than required precision and linking number 
         may be wrong. A similar procedure is used for the computation
         of Writhe number.}
\label{tab:link:rec}
\small
\setlength{\columnsep}{0pt}
\begin{procedure}[H]
  \caption{buildAngle($\pp,\pp',\qq,\qq'$)}
  \SetKw{OR}{or}
  \SetKw{AND}{and}
  $\bm\alpha \assign \qq-\pp\;$;\quad
  $\bm\beta  \assign \qq-\pp'\;$;\quad
  $\bm\gamma \assign \qq'-\pp'\;$;\quad
  $\bm\omega \assign \qq'-\pp\;$\;
  $\ell_\alpha \assign \norm{\bm\alpha}\;$;\quad
  $\ell_\beta  \assign \norm{\bm\beta}\;$;\quad
  $\ell_\gamma \assign \norm{\bm\gamma}\;$;\quad
  $\ell_\omega \assign \norm{\bm\omega}\;$\;
  $\bm\alpha \assign \bm\alpha/\ell_\alpha\;$;\quad
  $\bm\beta  \assign \bm\beta/\ell_\beta\;$;\quad
  $\bm\gamma \assign \bm\gamma/\ell_\gamma\;$;\quad
  $\bm\omega \assign \bm\omega/\ell_\omega\;$\;
  $\bm{t}_1  \assign \bm\alpha+\bm\gamma\;$;\quad
  $\bm{t}_2  \assign \bm\alpha\CROSS\bm\gamma\;$;\quad
  $t_3       \assign 1 + \bm\alpha\DOT\bm\gamma\;$\;
  $x_1       \assign t_3+\bm\beta\DOT\bm{t}_1\;$;\quad
  $y_1       \assign \bm\beta\DOT\bm{t}_2\;$;\quad
  $x_2       \assign t_3+\bm\omega\DOT\bm{t}_1\;$;\quad
  $y_2       \assign \bm\omega\DOT\bm{t}_2$\;
  $x         \assign x_1x_2+y_1y_2\;$;\quad
  $y         \assign x_1y_2-y_1x_2\;$\;
  \leIf{$y>0$ \OR ($y=0$ \AND $x<0$)}{$s\assign 1$}{$s\assign -1$}
  \leIf{$s\,y_1>0$}{$\sigma\assign-s$}{$\sigma\assign 0$\hfill// Compute error bound using Lemma~\ref{lem:13}}
  $R\assign\sqrt{x_1^2+y_1^2};\quad R'\assign\sqrt{x_2^2+y_2^2};\quad
  e\assign 2.829 + 57.516\left(1/R+1/R'\right)$\;
  \Return{$x,y,\sigma,s,\color{blue}e$}
\end{procedure}
\vspace{-0.8em}
\begin{procedure}[H]
  \caption{link($\pp,\qq$)}
  \SetKwFunction{buildangle}{buildangle}
  \SetKwFunction{addangle}{addangle}
  \SetKwFunction{scaleangle}{scaleangle}
  \SetKw{OR}{or}
  \SetKw{AND}{and}
  $[X,Y,\ell]\assign [1,0,0];\quad S\assign-1;\quad I\assign 0;\quad R\assign 0;\quad I_{\max}\assign [\pi/(2\umach)]$\;
  \For{$i=1,2,\ldots,N$}{
    \For{$j=1,2,\ldots,M$}{
      $x,y,\sigma,s,{\color{blue}e}\assign\buildAngle(\pp_i,\pp_{i+1},\qq_j,\qq_{j+1})$\;
      $x'\assign Xx-Yy\;$;\quad
      $y'\assign Xy+Yx\;$;\quad
      $\ell\assign\ell+\sigma\;$\;
      \leIf{$y'>0$ \OR ($y'=0$ \AND $x'<0$)}{$s'\assign 1$}{$s'\assign -1$}
      \lIf{$Ss>0$ \AND $Ss'<0$}{$\ell\assign\ell-s'$}      
      decompose $\cases{ x'=m\cdot 2^e & \cr y'=n\cdot 2^f&}$ with
                $\cases{\abs{m}\leq 1 & \cr \abs{n}\leq 1}$
      and set $h\assign\min\{e,f\}$\;
      $X\assign m\cdot 2^{e-h}\;$;\quad
      $Y\assign n\cdot 2^{f-h}\;$;\quad
      $S\assign s'\;$\;
       \tcp{Accumulate error bound as $\textrm{``Error Upper Bound''}\leq (I+R)\umach$}
      $E\assign e+2.829+R;$\quad
      $I\assign I+\textrm{floor}(E);$\quad
      $R\assign E-\textrm{floor}(E)$\;
      \lIf{$I\geq I_{\max}$}{Computation failed, linking number may be wrong!}
    }
  }
  \Return{$\ell$}
\end{procedure}
\vspace{-0.8em}
\begin{procedure}[H]
  \caption{writhe($\pp$)}
  \SetKwFunction{buildangle}{buildangle}
  \SetKwFunction{addangle}{addangle}
  \SetKwFunction{scaleangle}{scaleangle}
  \SetKw{OR}{or}
  \SetKw{AND}{and}
  $[X,Y,\ell]\assign [1,0,0];\quad S\assign-1;\quad E\assign 0$\;
  \For{$i=1,2,\ldots,N$}{
    \For{$j=i+2,i+3,\ldots,N$}{
      $x,y,\sigma,s,{\color{blue}e}\assign\buildAngle(\pp_i,\pp_{i+1},\pp_j,\pp_{j+1})$\;
      $x'\assign Xx-Yy\;$;\quad
      $y'\assign Xy+Yx\;$;\quad
      $\ell\assign\ell+\sigma\;$\;
      \leIf{$y'>0$ \OR ($y'=0$ \AND $x'<0$)}{$s'\assign 1$}{$s'\assign -1$}
      \lIf{$Ss>0$ \AND $Ss'<0$}{$\ell\assign\ell-s'$}      
      decompose $\cases{ x'=m\cdot 2^e & \cr y'=n\cdot 2^f&}$ with
                $\cases{\abs{m}\leq 1 & \cr \abs{n}\leq 1}$
      and set $h\assign\min\{e,f\}$\;
      $X\assign m\cdot 2^{e-h}\;$;\quad
      $Y\assign n\cdot 2^{f-h}\;$;\quad
      $E\assign E+e+2.829$\;
    }
  }
  \Return{$[2(\ell+\ARCTAN(Y,X)),2E\umach]$}
\end{procedure}
\end{table}

The algorithm of Table~\ref{tab:link:rec} which compute the linking number with a rigorous
upper bound of the error is build using Lemma~\ref{lem:13} and Lemma~\ref{lem:14}.
If the upper bound is less than $1/2$ then
the computation of linking number is exact.

\begin{remark}
   The constants $C_{ij}$ of equation~\eqref{eq:strict:cond} can be bounded as
   \begin{EQ}
      C_{ij} \geq \min_{i=1}^N\min_{j=1}^M\norm{\pp_i-\qq_j} = d
   \end{EQ}
   where $d$ is the minimum distance between the polygons.
   If $\ell$, the maximum segment length 
   \begin{EQ}
      \ell = \max\left\{
        \max_{i=1}^N\norm{\pp_i-\pp_{i+1}},
        \max_{j=1}^M\norm{\qq_j-\qq_{j+1}}
        \right\}
   \end{EQ}
   satisfy $\ell\leq d$, the conditions of Lemma~\ref{lem:14} are satisfied for all $i$ and $j$
   then the total error due to angle approximation and summation from Lemma~\ref{lem:8} becomes
   \begin{EQ}
      \textrm{Err} \leq NM 117.861 \umach+(NM-1)2.829\umach \leq  NM\,120.690 \umach.
   \end{EQ}
   To ensure that computation
   produce a correct linking number upper bound of the error must be less than $\pi/2$. Setting
   \begin{EQ}
      NM < \dfrac{\pi}{2\cdot 120.690\umach}\leq\dfrac{0.0131}{\umach}
   \end{EQ}
   if computation is performed using double precision floating point arithmetics,
   for example IEEE 754 with $\umach\approx 2.22\,10^{-16}$, and $N\approx M$ then
   the maximum number of segment by curve is $N \approx M \approx 76000000$
   while using single precision $\umach\approx 1.19\,10^{-7}$ and $N \approx M \approx 330$.
   However this estimation are upper bound, in practice linking number
   is computed exactly for curves with more segments.
\end{remark}

\section{Extend the computation of linking number for chains}

Given points $\{\pp_k\}_{k=1}^N$, a segment is a pair $e_{ij}=[\pp_i,\pp_j]$
and a chain is a formal summation
\begin{EQ}
  \mathcal{C} = \sum_{(i,j)\in \mathcal{E}} w_{ij} e_{ij}, \qquad w_{ij}\in\mathbbm{Z}
\end{EQ}
The operator $D$ is a linear operator such that $D e_{ij} = D[\pp_i,\pp_j]=\pp_i-\pp_j$.
A generalized closed loop is a chain with no border, i.e. $D\mathcal{C}=\bm{0}$ where
\begin{EQ}
  D\mathcal{C} 
  = D\sum_{(i,j)\in \mathcal{E}} w_{ij} e_{ij}
  = \sum_{(i,j)\in \mathcal{E}} w_{ij} D e_{ij}
  =\sum_{(i,j)\in \mathcal{E}} w_{ij} (\pp_i-\pp_j)
\end{EQ}
The linking number of two generalized closed loops
   $\mathcal{L}_1 = \sum_{(i,j)\in \mathcal{E}_1} w^{(1)}_{ij}e_{ij}$ and
   $\mathcal{L}_2 = \sum_{(k,l)\in \mathcal{E}_2} w^{(2)}_{kl}e_{kl}$
is defined as
\begin{EQ}
  L(\mathcal{L}_1 ,\mathcal{L}_2)=
   \sum_{(i,j)\in \mathcal{E}_1}
   \sum_{(k,l)\in \mathcal{E}_2}w^{(1)}_{ij}w^{(2)}_{kl}T(\pp_i,\pp_j,\pp_k,\pp_l)
\end{EQ}
as for the simple linking number $T(\pp_i,\pp_j,\pp_k,\pp_l) = \ANGLE{[x_{ijkl},y_{ijkl},\sigma_{ijkl}]}$
and 
\begin{EQ}
  L(\mathcal{L}_1 ,\mathcal{L}_2)=
   \ANGLE{ \bigoplus_{(i,j)\in \mathcal{E}_1}
           \bigoplus_{(k,l)\in \mathcal{E}_2}w^{(1)}_{ij}w^{(2)}_{kl}[x_{ijkl},y_{ijkl},\sigma_{ijkl}]}
\end{EQ}
this formula permits to extend the numerical computation of linking number
for a couple of general closed chain with integer weights.
The extension of the upper bound for error is trivial and not 
discussed further.

\appendix
\section{proofs of lemmata}

For the floating point arithmetics the following lemma's are true.

\begin{lemma}\label{lem:5}
  Given $2n$ floating point numbers $a_i$ and $b_i$ with $i=1,2,\ldots,n$
  and $n\umach<1$, then, the following estimates are true
  \begin{enumerate}
  \item
    $\FLOAT{a_1+a_2+\cdots+a_n} 
     = (a_1+a_2+\cdots+a_n)(1+\delta)$, \\
     with $\abs{\delta} \leq n\umach/(1-n\umach)$;
  \item
    $\FLOAT{a_1b_2+a_2b_2+\cdots+a_nb_n} 
     = (a_1b_2+a_2b_2+\cdots+a_nb_n)(1+\delta)$, \\
     with $\abs{\delta} \leq n\umach/(1-n\umach/2)$;
  \item
    $\FLOAT{a_1b_2+a_2b_2+\cdots+a_nb_n} 
     =a_1b_2+a_2b_2+\cdots+a_nb_n+\delta$, \\ where 
     $\abs{\delta}\leq
     n\umach(\abs{a_1}\abs{b_2}+\abs{a_2}\abs{b_2}+\cdots+\abs{a_nb_n})$;
  \item 
    $\FLOAT{a_1^2+a_2^2+\cdots+a_n^2} 
     = (a_1^2+a_2^2+\cdots+a_n^2)(1+\delta)$, \\
     with $\abs{\delta}\leq n\umach$.
  \end{enumerate}
\end{lemma}
\begin{proof}
  See reference~\cite{Higham:2002} for point 2 and \cite{Jeannerod:2013} Theorem 4.2
  for point 3.
  Point 4 follows directly from point 3.
\end{proof}
\begin{corollary}\label{cor:6}
  Given $n$ floating point numbers $a_i$ with $i=1,2,\ldots,n$
  and $n\umach<1$, then, the following estimate is true
  \begin{EQ}[rcll]
     \FLOAT{a_1^2+a_2^2+\cdots+a_n^2}
     &=& (a_1^2+a_2^2+\cdots+a_n^2)(1+\delta),
     \qquad &
     \abs{\delta}\leq n\umach
     \\
     \FLOAT{\sqrt{a_1^2+a_2^2+\cdots+a_n^2}}
     &=& \sqrt{a_1^2+a_2^2+\cdots+a_n^2}\,(1+\delta),
     \qquad &
     \abs{\delta}\leq\left(\sqrt{2}+\dfrac{n+1}{2}\right)\umach
     \\
  \end{EQ}
\end{corollary}
\begin{proof}
  The first equation is trivial. The second derive from the 
  Taylor expansion of $f(\umach)=(1+\umach)\sqrt{1+n\umach}$
  \begin{EQ}[rcl]
     \abs{f(\umach)-f(0)}&=&
     \abs{(1+\umach)\sqrt{1+n\umach}-1}
     = \abs{f'(z)\umach}
     = \abs{\sqrt{1+nz}+\dfrac{n(1+z)}{2\sqrt{1+nz}}}\umach\\
     &\leq& \left(\sqrt{2}+\dfrac{n+1}{2}\right)\umach,\qquad z\in(0,\umach)
  \end{EQ}
  \qed
\end{proof}

\begin{lemma}\label{lem:7}
  Giving the triples $[x,y,\sigma]$ and $[x',y',\sigma']$
  such that $\abs{x-x'}^2+\abs{y-y'}^2 \leq \delta^2$
  then the corresponding angles $\theta$ and $\theta'$ satisfy
  \begin{EQ}\label{eq:R:stima}
     \abs{\theta-\theta'}\leq \dfrac{\delta}{R}\left(1+\dfrac{\delta^2}{6R^2}\right)
  \end{EQ}
\end{lemma}
\begin{proof}
  From Figure~\ref{fig:Qregion} the difference $ \abs{\theta-\theta'}\leq \alpha$
  where $\alpha$ solve the geometric problem
  \begin{EQ}
     d \tan\alpha = \delta, \qquad x^2+y^2-d^2 = \delta^2
  \end{EQ}
  and thus solving for $\alpha$ with $R=\sqrt{x^2+y^2}$
  and using Taylor expansion
  \begin{EQ}
     \alpha = \ARCTAN\left(
        \dfrac{\delta}{\sqrt{R^2-\delta^2}}
     \right)
     =\dfrac{\delta}{R}+\dfrac{z^3}{6R^3},
     \qquad z\in(0,\delta)
  \end{EQ}
  and inequality follows trivially.
  \qed
\end{proof}

Giving the triples $[x,y,\sigma]$ and $[x',y',\sigma']$
using floating point arithmetics the approximate triple 
$[\tx'', \ty'',\tsigma'']$
satisfy
\begin{EQ}\label{eq:recurrence:fp}
   \tx'' = 2^{d}\,\FLOAT{xx'-yy'},\qquad
   \ty'' = 2^{d}\,\FLOAT{xy'+yx'},
\end{EQ}
where $d\in\mathbbm{Z}$ is chosen in such a way
\begin{EQ}\label{eq:XY:norm}
  \dfrac{1}{2} < \sqrt{(\tx'')^2+(\ty'')^2} \leq 1.
\end{EQ}
Multiplication and division by power of $2$ is done
\emph{without error} in floating point arithmetics (unless in case of overflow or underflow).
The scaling by $2^{d}$ is necessary to maintain
all the triples $[x,y,\sigma]$ in the summation which satisfy $1/2<\sqrt{x^2+y^2}\leq 1$.

\begin{lemma}\label{lem:8}
Let $[x,y,\sigma]$ and $[x',y',\sigma']$ two triple which satisfy
\begin{EQ}
  \frac{1}{2} \leq \sqrt{x^2+y^2} \leq 1,\qquad
  \frac{1}{2} \leq \sqrt{(x')^2+(y')^2} \leq 1,
\end{EQ}
and $\theta$ and $\theta'$ the corresponding angles.
Then, the summation of the triples with floating point arithmetics
\begin{EQ}
   [\tx'',\ty'',\tsigma''] = \FLOAT{[x,y,\sigma]\oplus [x',y',\sigma']},
   \qquad
   [x'',y'',\sigma''] = [x,y,\sigma]\oplus [x',y',\sigma'],
\end{EQ}
with the corresponding angles $\ttheta'' = \ANGLE{[\tx'',\ty'',\tsigma'']}$ and 
$\theta'' = \ANGLE{[x'',y'',\sigma'']}$ satisfy
\begin{EQ}
  \abs{\theta''-\ttheta''} \leq 2.829 \umach
\end{EQ}
\end{lemma}
\begin{proof}
From~\eqref{eq:recurrence:fp} using Lemma~\ref{lem:5} point 3
inequality $2ab \leq a^2+b^2$ with $x^2+y^2\leq 1$ and $(x')^2+(y')^2\leq 1$ 
\begin{EQ}[rcl]
   \hx'' &=& \FLOAT{xx'-yy'}
   =xx'-yy'+\delta_1,\\
   \hy'' &=& \FLOAT{xy'+yx'}
   =xy'+yx'+\delta_2,
\end{EQ}
where
\begin{EQ}[rcl]
   \abs{\delta_1} &\leq & 2\umach\left(\abs{x}\abs{x'}+\abs{y}\abs{y'}\right)
   \leq \umach\left(x^2+(x')^2+y^2+(y')^2\right) \leq 2\umach,
   \\
   \abs{\delta_2} &\leq & 2\umach\left(\abs{x}\abs{y'}+\abs{x}\abs{y'}\right)
   \leq \umach\left(x^2+(x')^2+y^2+(y')^2\right) \leq 2\umach,
\end{EQ}
and thus
\begin{EQ}
   \abs{\hx''-x''}^2+
   \abs{\hy''-y''}^2
   \leq \abs{\delta_1}^2 + \abs{\delta_2}^2 \leq 8\umach^2
\end{EQ}
thus the distance of $(x'',y'')$ the exact point on the triple $[x'',y'',\sigma'']$
with $(\hx'',\hy'')$ the approximate point on the triple $[\hx'',\hy'',\tsigma'']$
which is equivalent to the triple  $[\tx'',\ty'',\sigma'']$ is less than
$2\sqrt{2}\umach$.
Using estimate~\eqref{eq:R:stima} of Lemma~\ref{lem:7} with $R\geq 1/2$
\begin{EQ}
  \abs{\theta-\theta'}\leq 2\sqrt{2}\,\umach\left(1+\dfrac{8}{3}\umach^2\right)
\end{EQ}
The conclusion follows trivially from $\umach\leq 10^{-7}$ by assumption~\ref{ass:1}.
\qed
\end{proof}
\begin{lemma}\label{lem:9}
  Given three floating point numbers $a$, $b$ and $c$ and the
  normalized floating point numbers $A$, $B$ and $C$ computed 
  with the following algorithm
  \begin{EQ}
     d = \FLOAT{\sqrt{a^2+b^2+c^2}},\qquad
     A = \FLOAT{a/d},\qquad
     B = \FLOAT{b/d},\qquad
     C = \FLOAT{b/d},
  \end{EQ}
  satisfy
  \begin{EQ}\label{eq:ABC}
     A=\dfrac{a(1+\delta_A)}{\sqrt{a^2+b^2+c^2}}
     \qquad
     B=\dfrac{b(1+\delta_B)}{\sqrt{a^2+b^2+c^2}}
     \qquad
     C=\dfrac{c(1+\delta_C)}{\sqrt{a^2+b^2+c^2}}
  \end{EQ}
  where $\abs{\delta_A}$, $\abs{\delta_B}$ and $\abs{\delta_C}$ are less than $3.415\umach$.
\end{lemma}
\begin{proof}
  From corollary~\ref{cor:6} it follows $d = \sqrt{a^2+b^2+c^2}(1+\delta)$
  with $\abs{\delta}\leq (2+\sqrt{2})\umach$
  and
  \begin{EQ}
     A = \FLOAT{\dfrac{a}{d}} = \dfrac{a}{d}(1+\delta')
     = \dfrac{a}{\sqrt{a^2+b^2+c^2}}\dfrac{1+\delta'}{1+\delta}
     \qquad
    \abs{\delta}\leq (2+\sqrt{2})\umach, \quad \abs{\delta'}\leq \umach
  \end{EQ}
  moreover by Assumption~\ref{ass:1}
  \begin{EQ}
     \abs{\dfrac{1+\delta'}{1+\delta}-1}
     =
     \abs{\dfrac{\delta-\delta'}{1+\delta}}
     \leq
     \dfrac{3+\sqrt{2}}{1-\umach}\umach \leq 3.415 \umach
  \end{EQ}
  and, thus, $\abs{\delta_A}\leq 3.415 \umach$. The rest of the Lemma follow easily.
\qed
\end{proof}

\begin{lemma}\label{lem:10}
  Given four vector $\bm{p}$, $\bm{q}$, $\bm{r}$ and $\bm{s}$ in $\mathbbm{R}^3$
  and the following algorithm evaluated using floating point arithmetics
  \begin{EQ}
     \bm{a} = \FLOAT{\bm{p}-\bm{q}}, \qquad
     \bm{b} = \FLOAT{\bm{r}-\bm{s}}, \qquad
     \tilde{\bm{a}} = \FLOAT{\dfrac{\bm{a}}{\norm{\bm{a}}}}, \qquad
     \tilde{\bm{b}} = \FLOAT{\dfrac{\bm{b}}{\norm{\bm{b}}}}, \qquad
     c = \FLOAT{\tilde{\bm{a}}\DOT\tilde{\bm{b}}},
  \end{EQ}
  then
  \begin{EQ}\label{eq:ineq}
     \abs{c-
           \frac{\bm{p}-\bm{q}}{\norm{\bm{p}-\bm{q}}}\DOT\frac{\bm{r}-\bm{s}}{\norm{\bm{r}-\bm{s}}}}
     \leq 13.838 \umach
  \end{EQ}
\end{lemma}
\begin{proof}
  Using assumption~\ref{ass:1} for floating point arithmetics with Lemma~\ref{lem:9}
  \begin{EQ}[rclrcl]\label{eq:ab:1}
     \bm{a} &=& (\bm{I}+\bm{D}_1)(\bm{p}-\bm{q}), \qquad &
     \bm{b} &=& (\bm{I}+\bm{D}_2)(\bm{r}-\bm{w}), \\
     \tilde{\bm{a}} &=& (\bm{I}+\bm{D}_3)\dfrac{(\bm{I}+\bm{D}_1)(\bm{p}-\bm{q})}
                            {\norm{(\bm{I}+\bm{D}_1)(\bm{p}-\bm{q})}}, \qquad &
     \tilde{\bm{b}} &=& (\bm{I}+\bm{D}_4)\dfrac{(\bm{I}+\bm{D}_2)(\bm{r}-\bm{w})}
                            {\norm{(\bm{I}+\bm{D}_2)(\bm{r}-\bm{w})}},
  \end{EQ}
  where $\norm{\bm{D}_k}_{\infty}\leq \umach$ for $k=1,2$ 
  and $\norm{\bm{D}_k}_{\infty}\leq 3.415 \umach$ for $k=3,4$ 
  are diagonal matrices.
  For a generic vector $\bm{v}$ and diagonal matrix $\bm{D}$ with $\norm{\bm{D}}_\infty \leq \umach$
  \begin{EQ}[rcl]\label{eq:ab:2}
     \norm{\bm{v}}-\norm{\bm{D}}_\infty\norm{\bm{v}}
     \leq
     \norm{\bm{v}}-\norm{\bm{D}\bm{v}}
     \leq
     &
     \norm{\bm{v}+\bm{D}\bm{v}} 
     &
     \leq 
     \norm{\bm{v}}+\norm{\bm{D}\bm{v}} 
     \leq
     \norm{\bm{v}}+\norm{\bm{D}}_\infty\norm{\bm{v}} 
     \\
    1-\norm{\bm{D}}_\infty\leq
    \dfrac{\norm{\bm{v}}}{\norm{\bm{v}}+\norm{\bm{D}}_\infty\norm{\bm{v}}}
    \leq
    &
    \dfrac{\norm{\bm{v}}}{\norm{\bm{v}+\bm{D}\bm{v}}}
    &
    \leq \dfrac{\norm{\bm{v}}}{\norm{\bm{v}}-\norm{\bm{D}}_\infty\norm{\bm{v}}}
    \leq \dfrac{1}{1-\norm{\bm{D}}_\infty}
  \end{EQ}
  and thus $\norm{\bm{v}}/\norm{\bm{v}+\bm{D}\bm{v}}=1+\delta$
  with $\delta \leq 1.001 \umach$.
  Using~\eqref{eq:ab:1} and~\eqref{eq:ab:2}
  \begin{EQ}
     \tilde{\bm{a}} = \dfrac{\norm{\bm{p}-\bm{q}}}
                            {\norm{(\bm{I}+\bm{D}_1)(\bm{p}-\bm{q})}}(\bm{I}+\bm{D}_5)
                            \dfrac{\bm{p}-\bm{q}}{\norm{\bm{p}-\bm{q}}}
     = (\bm{I}+\bm{D}_6)\dfrac{\bm{p}-\bm{q}}{\norm{\bm{p}-\bm{q}}},
  \end{EQ}
  where
  \begin{EQ}[ll]
     \bm{D}_5 = \bm{D}_1+\bm{D}_3+\bm{D}_1\bm{D}_3
     \quad &
     \norm{\bm{D}_5}_\infty \leq \umach(4.415+3.415\umach) \leq 4.416 \umach
     \\  
     \bm{D}_6 = \delta\bm{I}+(1+\delta)\bm{D}_5
     \quad &
     \norm{\bm{D}_6}_\infty \leq \umach(5.417+4.417\umach) \leq 5.418 \umach
     \\  
  \end{EQ}
  and analogously
  \begin{EQ}
     \tilde{\bm{b}}
     = (\bm{I}+\bm{D}_7)\dfrac{\bm{r}-\bm{s}}{\norm{\bm{r}-\bm{s}}},
     \qquad 
     \norm{\bm{D}_7}_\infty \leq 5.418 \umach.
  \end{EQ}
  If $\bm{D}$ is a diagonal matrix and $\bm{v}$ and $\bm{w}$ generic vector,
  from H\"older and Cauchy-Schwarz inequalities it is easy to verify
  \begin{EQ}\label{eq:CS}
    \abs{\bm{v}^T\bm{D}\bm{w}} \leq \norm{\bm{D}}_{\infty} \norm{\bm{v}}\norm{\bm{w}}
  \end{EQ}
  Let $\bm{v}=\dfrac{\bm{p}-\bm{q}}{\norm{\bm{p}-\bm{q}}}$ and 
  $\bm{w}=\dfrac{\bm{r}-\bm{s}}{\norm{\bm{r}-\bm{s}}}$ then using~\eqref{eq:CS}
  \begin{EQ}[rcl]\label{eq:dot:v:w}
    \tilde{\bm{a}}\DOT\tilde{\bm{b}} &=&
     \left[(\bm{I}+\bm{D}_6)\bm{v}\right] \DOT \left[(\bm{I}+\bm{D}_7)\bm{w}\right]
     = \bm{v}\DOT\bm{w} + \bm{v}^T(\bm{D}_6+\bm{D}_7)\bm{w}
     +\bm{v}^T(\bm{D}_6\bm{D}_7)\bm{w} \\
    &=& \bm{v}\DOT\bm{w}+\delta_1, \qquad \abs{\delta_1} \leq 10.837 \umach
  \end{EQ}
  and
  \begin{EQ}[rcll]
    \norm{\tilde{\bm{a}}} &=& \norm{(\bm{I}+\bm{D}_6)\dfrac{\bm{p}-\bm{q}}{\norm{\bm{p}-\bm{q}}}}
    =\norm{\bm{I}+\bm{D}_6} = 1+\delta_2,\qquad &\abs{\delta_2}\leq 5.418 \umach
    \\
    \norm{\tilde{\bm{b}}} &=& \norm{(\bm{I}+\bm{D}_7)\dfrac{\bm{r}-\bm{s}}{\norm{\bm{r}-\bm{s}}}}
    =\norm{\bm{I}+\bm{D}_7} = 1+\delta_3,\qquad &\abs{\delta_3}\leq 5.418 \umach
  \end{EQ}
  From Lemma~\ref{lem:5} point 3 and~\eqref{eq:dot:v:w} with Cauchy-Schwarz inequality
  \begin{EQ}
     c = \tilde{\bm{a}}\DOT\tilde{\bm{b}}+\delta_4,\qquad \abs{\delta_4} 
        \leq 3\umach \abs{\tilde{\bm{a}}}\DOT\abs{\tilde{\bm{b}}} \leq 3\umach 
        \norm{\tilde{\bm{a}}}\norm{\tilde{\bm{b}}}
        \leq 3.001\umach
  \end{EQ}
  using~\eqref{eq:dot:v:w}
  $   c = \frac{\bm{p}-\bm{q}}{\norm{\bm{p}-\bm{q}}}\DOT\frac{\bm{r}-\bm{s}}{\norm{\bm{r}-\bm{s}}}
     +\delta_1+\delta_4$,
  and the inequality~\eqref{eq:ineq} follows easily.
\qed
\end{proof}

\begin{lemma}\label{lem:11}
  Given six vector $\bm{p}_1$, $\bm{p}_2$, $\bm{q}_1$, $\bm{q}_2$, $\bm{r}_1$ and $\bm{r}_2$ in $\mathbbm{R}^3$
  and the following algorithm evaluated using floating point arithmetics
  \begin{EQ}
     \bm{a} = \FLOAT{\bm{p}_1-\bm{p}_2}, \qquad
     \bm{b} = \FLOAT{\bm{q}_1-\bm{q}_2}, \qquad
     \bm{c} = \FLOAT{\bm{r}_1-\bm{r}_2}, \\
     \tilde{\bm{a}} = \FLOAT{\dfrac{\bm{a}}{\norm{\bm{a}}}}, \qquad
     \tilde{\bm{b}} = \FLOAT{\dfrac{\bm{b}}{\norm{\bm{b}}}}, \qquad
     \tilde{\bm{c}} = \FLOAT{\dfrac{\bm{c}}{\norm{\bm{c}}}}, \\
     d_1 = \FLOAT{\tilde{\bm{a}}\DOT\tilde{\bm{b}}},\qquad
     d_2 = \FLOAT{\tilde{\bm{b}}\DOT\tilde{\bm{c}}},\qquad
     d_3 = \FLOAT{\tilde{\bm{a}}\DOT\tilde{\bm{c}}},\qquad
     d_4 = \FLOAT{1+d_1+d_2+d_3},
  \end{EQ}
  then if $t_4$ is defined as
  \begin{EQ}[rclrcl]
     t_1&=&\frac{\bm{p}_1-\bm{p}_2}{\norm{\bm{p}_1-\bm{p}_2}}\DOT\frac{\bm{q}_1-\bm{q}_2}{\norm{\bm{q}_1-\bm{q}_2}},
     \qquad &
     t_2&=&\frac{\bm{p}_1-\bm{p}_2}{\norm{\bm{p}_1-\bm{p}_2}}\DOT\frac{\bm{r}_1-\bm{r}_2}{\norm{\bm{r}_1-\bm{r}_2}},
     \\
     t_3&=&\frac{\bm{q}_1-\bm{q}_2}{\norm{\bm{q}_1-\bm{q}_2}}\DOT\frac{\bm{r}_1-\bm{r}_2}{\norm{\bm{r}_1-\bm{r}_2}},
     \qquad&
     t_4 &=& 1+t_1+t_2+t_3,
  \end{EQ}
  then
  \begin{EQ}\label{eq:ineq:lem:11}
     \abs{d_4-t_4} \leq 57.515\umach
  \end{EQ}
\end{lemma}
\begin{proof}
  From Lemma~\ref{lem:10} $d_k = t_k(1+\delta_k)$ with $\abs{\delta_k}\leq 13.838 \umach$
  and from Lemma~\ref{lem:5} point 1
  \begin{EQ}[rcl]
     \FLOAT{1+d_1+d_2+d_3} 
     &=&
     \FLOAT{1+t_1(1+\delta_1)+t_2(1+\delta_2)+t_3(1+\delta_3)}
     \\
     &=&\big(1+t_1(1+\delta_1)+t_2(1+\delta_2)+t_3(1+\delta_3)\big)(1+\delta_4)
     \\
     &=&\big(1+t_1+t_2+t_3+t_1\delta_1+t_2\delta_2+t_3\delta_3\big)(1+\delta_4)
     \\
     &=&\big(1+t_1+t_2+t_3+\delta_4+
                t_1(\delta_1+\delta_4)+ \\
     &&           t_2(\delta_2+\delta_4)+
                t_3(\delta_3+\delta_4)+
     \big(t_1\delta_1+t_2\delta_2+t_3\delta_3\big)\delta_4
     \big)
     \\
     &=&\left(1+t_1+t_2+t_3+\delta_5\right)
  \end{EQ}
  where $\abs{\delta_4}\leq 4\umach/(1-2\umach)$ and
  \begin{EQ}[rcl]
     \abs{\delta_5} &\leq&
     \abs{\delta_4}(1+\abs{\delta_1}+\abs{\delta_2}+\abs{\delta_3})+
     \abs{\delta_1+\delta_4}+
     \abs{\delta_2+\delta_4}+
     \abs{\delta_3+\delta_4} \\
     &\leq&
     \dfrac{4\umach}{1-2\umach}(4+41.514\umach) + 41.514\umach
  \end{EQ}
  And using assumption $\umach\leq 10^{-7}$ inequality~\eqref{eq:ineq:lem:11} follows.\qed
\end{proof}

\begin{lemma}\label{lem:12}
  Given six vector $\bm{p}$, $\bm{q}$, $\bm{r}$, $\bm{s}$, $\bm{t}$ and $\bm{u}$ in $\mathbbm{R}^3$
  and the following algorithm evaluated using floating point arithmetics
  \begin{EQ}
     \bm{a} = \FLOAT{\bm{p}-\bm{q}}, \quad
     \bm{b} = \FLOAT{\bm{r}-\bm{s}}, \quad
     \bm{c} = \FLOAT{\bm{t}-\bm{u}}, \\
     \tilde{\bm{a}} = \FLOAT{\dfrac{\bm{a}}{\norm{\bm{a}}}}, \quad
     \tilde{\bm{b}} = \FLOAT{\dfrac{\bm{b}}{\norm{\bm{b}}}}, \quad
     \tilde{\bm{c}} = \FLOAT{\dfrac{\bm{c}}{\norm{\bm{c}}}}, \quad
     \tilde{\bm{d}} = \FLOAT{\tilde{\bm{b}}\CROSS\tilde{\bm{c}}},\quad
     e = \FLOAT{\tilde{\bm{a}}\DOT\tilde{\bm{d}}},
  \end{EQ}
  then
  \begin{EQ}\label{eq:ineq:1}
     \abs{e-\frac{\bm{p}-\bm{q}}{\norm{\bm{p}-\bm{q}}}\DOT\left(
           \frac{\bm{r}-\bm{s}}{\norm{\bm{r}-\bm{s}}}\CROSS\frac{\bm{t}-\bm{u}}{\norm{\bm{t}-\bm{u}}}
           \right)
           }
     \leq 23.26 \umach
  \end{EQ}
\end{lemma}
\begin{proof}
  Using the same arguments of Lemma~\ref{lem:10}
  \begin{EQ}
    \tilde{\bm{a}} = (\bm{I}+\bm{D}_1)\dfrac{\bm{p}-\bm{q}}{\norm{\bm{p}-\bm{q}}},
    \qquad
    \tilde{\bm{b}} = (\bm{I}+\bm{D}_2)\dfrac{\bm{p}-\bm{q}}{\norm{\bm{p}-\bm{q}}},
    \qquad  
    \tilde{\bm{c}} = (\bm{I}+\bm{D}_3)\dfrac{\bm{p}-\bm{q}}{\norm{\bm{p}-\bm{q}}},
  \end{EQ}
  and
  \begin{EQ}
    \max\big\{\norm{\bm{D}_1}_\infty,\norm{\bm{D}_2}_\infty,\norm{\bm{D}_3}_\infty\big\}
    \leq 5.418 \umach
  \end{EQ}
  from Lemma~\ref{lem:5} with $a$, $b$, $c$, $d$ such that $\abs{a}$, $\abs{b}$, $\abs{c}$,
  $\abs{d}$ less or equal that $1$:
  \begin{EQ}
    \FLOAT{ab-cd}=(ab-cd)(1+\epsilon_1),\qquad \abs{\epsilon_1}\leq \frac{2\umach}{1-\umach} \leq 2.001 \umach
  \end{EQ}
  and, thus,
  \begin{EQ}[rcll]
    \tilde{\bm{d}} &=&
    (\bm{I}+\bm{D}_4)\left(
    (\bm{I}+\bm{D}_2)\dfrac{\bm{p}-\bm{q}}{\norm{\bm{p}-\bm{q}}}
    \times
    (\bm{I}+\bm{D}_3)\dfrac{\bm{p}-\bm{q}}{\norm{\bm{p}-\bm{q}}}
    \right),
    \qquad
    &
    \norm{\bm{D}_4}_\infty\leq 2.001 \umach,
    \\
    \tilde{\bm{d}} &=&
    (\bm{I}+\bm{D}'_2)\dfrac{\bm{p}-\bm{q}}{\norm{\bm{p}-\bm{q}}}
    \times
    (\bm{I}+\bm{D}'_3)\dfrac{\bm{p}-\bm{q}}{\norm{\bm{p}-\bm{q}}},
    \qquad
    &
    \left\{
    \begin{array}{c}
    \norm{\bm{D}'_2}_\infty \\[0.4em]
    \norm{\bm{D}'_3}_\infty
    \end{array}
    \right\}
    \leq 7.420 \umach,
    \\
    \tilde{\bm{d}} &=&
    \dfrac{\bm{p}-\bm{q}}{\norm{\bm{p}-\bm{q}}}
    \times
    \dfrac{\bm{p}-\bm{q}}{\norm{\bm{p}-\bm{q}}}+\bm{E},
    \qquad
    &
    \norm{\bm{E}}_\infty
    \leq 14.841 \umach,
  \end{EQ}
  and using lemma Lemma~\ref{lem:5} for the scalar product inequality~\eqref{eq:ineq:1}
  follows.
\qed
\end{proof}

\begin{lemma}[A posteriori error bound]\label{lem:13}
  The accuracy of floating point computation of angle $\theta_{ij}$ satisfy
  \begin{EQ}
     \abs{\theta_{ij}-\ttheta_{ij}} \leq 
     \left( 2.829 + 
     \dfrac{57.516}{R}+
     \dfrac{57.516}{R'}
     \right) \umach,
     \\
     R = \FLOAT{\sqrt{x^2+y^2}},\quad
     R' = \FLOAT{\sqrt{(x')^2+(y')^2}},
  \end{EQ}
  where $R$ and $R'$ aere the length of vector $(x,y)$ and $(x',y')$ computed
  using floating point arithmetics.
\end{lemma}
\begin{proof}
From Lemma~\ref{lem:11} and~\ref{lem:12} the floating point computation
of $x$, $y$, $x'$ and $y'$ of equation \eqref{eq:xy} in Lemma~\ref{lem:3} satisfy:
\begin{EQ}
  \cases{
     x = \FLOAT{x}+\delta_x, & \\
     y = \FLOAT{y}+\delta_y, &
  }
  \;
  \cases{
     x' = \FLOAT{x'}+\delta_{x'}, & \\
     y' = \FLOAT{y'}+\delta_{y'}, &
  }
  \;
  \left\{
     \begin{array}{c}
     \abs{\delta_{x}} \\[0.4em]
     \abs{\delta_{x'}}
     \end{array}
  \right\}
  \leq 57.515 \umach,
  \;
  \left\{
     \begin{array}{c}
     \abs{\delta_{y}} \\[0.4em]
     \abs{\delta_{y'}}
     \end{array}
  \right\}
  \leq 20.257 \umach
\end{EQ}
Let be $\theta$, $\ttheta$,  $\theta'$ and $\ttheta'$ the angle corresponding to 
vectors $(x,y)$, $(\tx,\ty)$, $(x',y')$ and $(\tx',\ty')$, respectively.
From Lemma~\ref{lem:7}
\begin{EQ}[rcl]
   \abs{\theta-\ttheta}&\leq& \dfrac{57.515}{\sqrt{\FLOAT{x}^2+\FLOAT{x}^2}}\umach\left(1+\dfrac{551.3293}{\FLOAT{x}^2+\FLOAT{x}^2}\umach^2\right),
   \\
   \abs{\theta'-\ttheta'}&\leq& \dfrac{57.515}{\sqrt{\FLOAT{x'}^2+\FLOAT{y'}^2}}\umach\left(1+\dfrac{551.3293}{\FLOAT{x'}^2+\FLOAT{y'}^2}\umach^2\right)
\end{EQ}
and from corollary~\ref{cor:6}
\begin{EQ}
  \FLOAT{\sqrt{x^2+y^2}}
  = \sqrt{x^2+y^2}\,(1+\delta),
  \qquad
  \abs{\delta}\leq\left(\sqrt{2}+\dfrac{3}{2}\right)\umach \leq 2.9143 \umach
\end{EQ}
so that
\begin{EQ}
   \abs{\theta-\ttheta}\leq \dfrac{57.516}{\FLOAT{\sqrt{x^2+y^2}}}\umach,
   \qquad
   \abs{\theta'-\ttheta'}\leq \dfrac{57.516}{\FLOAT{\sqrt{(x')^2+(y')^2}}}\umach.
\end{EQ}
Angle $\theta_{ij}$ is the sum
of $\theta+\theta'$ performed using algorithm of definition~\ref{triple:plus:def},
and thus
\begin{EQ}
   \ttheta_{ij}=\ttheta+\ttheta'+\epsilon
\end{EQ}
where $\epsilon$ is the error due to the summation of the angles
represented as triple, that by Lemma~\ref{lem:8} is less than $2.829\umach$.\qed
\end{proof}
\begin{lemma}[A priory error bound]\label{lem:14}
  Let $\pp_i$, $\pp_{i+1}$, $\qq_j$ and $\qq_{j+1}$ such that~\eqref{eq:strict:cond} is satisfied
  then the error in the computation of $\Delta\Theta_{ij}$ is less than $117.861 \umach$.
\end{lemma}
\begin{proof}
  From equation~\eqref{eq:xy}
  $
    x=
    1+\talpha\DOT\tgamma+\talpha\DOT\tbeta+\tbeta\DOT\tgamma=
    1+\cos\theta_{\alpha,\gamma}+\cos\theta_{\alpha,\beta}+\cos\theta_{\beta,\gamma}
  $
  where $\theta_{\alpha,\gamma}$, $\theta_{\alpha,\beta}$ and $\theta_{\beta,\gamma}$
  are the angles between the respective vectors.
  The scalar product $\talpha\DOT\tgamma$ can take any value in $[-1,1]$, thus,
  to ensure $x\geq 1$ it is enough that $\talpha\DOT\tbeta\geq1/2$ and $\tbeta\DOT\tgamma\geq1/2$
  or $\theta_{\alpha,\beta}\leq\pi/3$ and $\theta_{\beta,\gamma}\leq\pi/3$.
  To ensure this last conditions it is enough that the named angles are 
  opposite to the minimum length edge of the triangle.
  This condition is ensured by inequality~\eqref{eq:strict:cond}.
  Similar argument follow for $x'$.
  If $x$ and $x'$ are greater that $1$ then $R$ and $R'$ of Lemma~\ref{lem:13} are 
  greater or equal to $1$ so that $\abs{\Delta\Theta_{ij}-\widetilde{\Delta\Theta_{ij}}}\leq 117.861 \umach$.
  \qed
\end{proof}

\bibliographystyle{spmpsci}      
\bibliography{LinkingNumber-bibliography} 

\end{document}